\documentclass[11pt, letterpaper, oneside,reqno]{amsart}

\usepackage{latexsym, amsmath, amssymb, amsfonts, amscd,bm}
\usepackage{amsthm}
\usepackage{t1enc}
\usepackage[mathscr]{eucal}
\usepackage{indentfirst}
\usepackage{graphicx, pb-diagram}
\usepackage{fancyhdr}
\usepackage{fancybox}
\usepackage{enumerate}
\usepackage{color}
\usepackage[all]{xy}
\usepackage[hidelinks]{hyperref}
\usepackage{mdframed}
\usepackage{tikz}
\usepackage{tikz-cd}
\usetikzlibrary{matrix,arrows}

\usepackage{url}

\theoremstyle{plain}
\newtheorem{thm}{Theorem}[section]
\newtheorem*{thm*}{Theorem}

\newtheorem{prop}[thm]{Proposition}

\newtheorem{lemma}[thm]{Lemma}
\newtheorem{cor}[thm]{Corollary}

\theoremstyle{definition}
\newtheorem{defi}[thm]{Definition}

\newtheorem{ep}[thm]{Example}

\theoremstyle{remark} 
\newtheorem{remark}[thm]{Remark}

\newcommand{\ZZ}{\ensuremath{\mathbb Z}}
\newcommand{\CC}{\ensuremath{\mathbb C}}
\newcommand{\RR}{\ensuremath{\mathbb R}}
\newcommand{\g}{\ensuremath{\mathfrak{g}}}

\newcommand{\h}{\ensuremath{\mathfrak{h}}}

\DeclareMathOperator{\gr}{graph}

\newcommand{\cL}{\mathcal{L}}
\newcommand{\cD}{\mathcal{D}}

\newcommand{\ra}{\rangle}
\newcommand{\la}{\langle}
\newcommand{\oo}{[\![}
\newcommand{\cc}{]\!]}

 \newcommand{\pp}{\varepsilon}

\title[Stability]{Stability of fixed points of Dirac structures
}

\author{Karandeep J. Singh}
\address{Department of Mathematics, KU Leuven, Celestijnenlaan 200B, 3001 Leuven, Belgium}
\email{\href{mailto:karandeep.singh@kuleuven.be}{\underline{\smash{karandeep.singh@kuleuven.be}}}}

\author{Marco Zambon}
\address{Department of Mathematics, KU Leuven, Celestijnenlaan 200B, 3001 Leuven, Belgium}
\email{\href{mailto:marco.zambon@kuleuven.be}{\underline{\smash{marco.zambon@kuleuven.be}}}}
  
\keywords{Stability, singular points,  $L_\infty$-algebra, Dirac geometry.}

\thanks{2020 Mathematics Subject Classification:  primary 17B70, 53D17, secondary 8H15.}

\begin{document}

\begin{abstract}
Given an $L_{\infty}$-algebra $V$ and an   $L_{\infty}$-subalgebra $W$, we give sufficient conditions for all small Maurer-Cartan elements of $V$  to be equivalent to  Maurer-Cartan elements lying in $W$.
As an application, we obtain a stability criterion for fixed points of a Dirac structure in an arbitrary Courant algebroid of split signature (for instance a  twisted Poisson structure), i.e. points where
the corresponding leaf is zero-dimensional. The criterion guarantees that any nearby Dirac structure also has a fixed point.
 \end{abstract}

\maketitle

\setcounter{tocdepth}{1} 
\tableofcontents

\section*{Introduction}

Stability questions appear naturally in mathematics. For instance, given a vector field $X$ vanishing at a point $p$, one can ask about the stability of $p$: does every vector field sufficently close to $X$ have a zero nearby $p$? 
Given a Lie algebra structure on a fixed vector space and a Lie subalgebra $\h$, one can ask about the stability of $\h$: does every sufficently close Lie algebra structure admit a Lie subalgebra nearby $\h$?

The main contribution of this paper is two-fold. First we state an algebraic theorem about  $L_{\infty}[1]$-algebras.
This   theorem can be applied to a variety of stability questions. In the second part of the paper we apply it to a specific geometric problem, obtaining a stability criterion for fixed points of Dirac structures. This includes twisted Poisson structures as a special case, and extends some  results obtained in \cite{CrFe}\cite{dufour2005stability}\cite{KarandeepStability}.
 
\bigskip
Recall that \emph{$L_{\infty}[1]$-algebras} are a notion equivalent to the $L_{\infty}$-algebras introduced by Lada and Stasheff in the 1990's \cite{LadaStasheff}, in order to provide a ``up to homotopy'' version of Lie algebras. They contain special elements -- called Maurer-Cartan elements 
--   which come equipped with an equivalence relation. Deformation problems are typically governed by such algebraic structures, in the sense that (equivalence classes of) deformations are parametrized by  (equivalence classes of) Maurer-Cartan  elements of the $L_{\infty}[1]$-algebra.

We paraphrase our main algebraic results as follows, omitting technical assumptions, and refer to theorem \ref{thm:mainthm} for the full statement:
\begin{thm*}
 Let $V$ be an $L_{\infty}[1]$-algebra, {whose underlying cochain complex we denote by by $(V,d)$}. Let $W$ be an $L_{\infty}[1]$-subalgebra of {finite codimension}, and fix a Maurer-Cartan element $Q$ of $W$.  
 Denote by $d^Q$ is the differential on $V$ obtained twisting $d$ by $Q$, 
  and view $d^Q$ as a differential on the quotient $V/W$.
If $$H^0(V/W, d^Q)=0,$$ 
then, {under some technical conditions}, 
any Maurer-Cartan element of $V$ sufficently close to $Q$ is  equivalent to a Maurer-Cartan element lying in $W$.
\end{thm*}
This result  extends a previous one on differential graded Lie algebras by the first author \cite[theorem 3.20]{KarandeepStability} 
 {(see also the works of Dufour-Wade \cite{dufour2005stability}
 and Crainic-Fernandes \cite{RigFlexPois}\cite{CrFe}).}

\bigskip
As an application of the above algebraic theorem, we consider \emph{Dirac structures} \cite{Cou}, geometric structures which include Poisson bivector fields and closed 2-forms, and which can be used to characterize Hitchin's generalized complex structures \cite{Hi}. 
Dirac structures are defined as lagrangian and involutive subbundles of \emph{Courant algebroids} \cite{LWX}. Important examples of the latter are $TM\oplus T^*M$, endowed with a bracket that depends on a choice of closed 3-form $H$ on the manifold $M$. In that case, one speaks of $H$-twisted Dirac structures  \cite{SW} (this includes $H$-twisted Poisson   structures). They first appeared in the context of $\sigma$-models in physics, in the work of Klim{\v c}ik-Strobl
\cite{KLIMCIK2002341} and Park \cite{ParkOpen}, and in that context $H$ is  called the Wess-Zumino-Witten $3$-form.

A byproduct of this note is a geometric characterization of the equivalences of Dirac structures induced by the $L_{\infty}[1]$-algebra governing deformations of Dirac structures: they are given by applying inner automorphisms of the ambient Courant algebroid, see proposition \ref{prop:CAauto}.
 
Applying the above algebraic theorem, and upon making explicit the assumptions and conclusions, one obtains theorem \ref{thm:stableDirac2}. This is a statement  on the stability of fixed points {(i.e. zero-dimensional leaves)} of Dirac structures .
We state a simplified version  as follows: 
\begin{thm*}
Let $E\to M$ be a Courant algebroid whose pairing has split signature, denote by $\rho\colon E\to TM$ its anchor map.
Let $A\subset E$ be a Dirac structure which has a fixed point at $p\in M$, i.e.  $\rho(A_p)=0$. Denote by $\g$ the Lie algebra $A_p$, and consider the Lie ideal $\mathfrak{h}:=(\ker(\rho|_{E_p}))^{\perp}$. 
Assume that
$$
H^2\left(\frac{\wedge^{\bullet} \g^*}{\wedge^{\bullet}\h^{\circ}},\overline{d_{\g}}\right) = 0.
$$
Then any Dirac structure   sufficiently close to $A$ admits
a fixed point nearby $p$, and lying in the leaf of $E$ through $p$.
 \end{thm*}
{This theorem clarifies and improves \cite[theorem 5.50]{KarandeepStability}, since it holds in wider generality and without making any auxiliary choice.} 
 
We conclude this note presenting some examples of the above theorem in \S \ref{sec:ex}.
For instance, on a  Lie group $G$ with a bi-invariant metric we  consider  the Cartan-Dirac structure, which is twisted by the Cartan-Dirac $3$-form $H$. The identity $e\in G$ is a fixed point of the Cartan-Dirac structure. If the second Lie algebra cohomology vanishes, then any sufficently close  $H$-twisted Dirac structure also has a fixed point nearby $e$.

 \bigskip
  \paragraph{\bf Acknowledgements.} 

We acknowledge partial support by
the FWO and FNRS under EOS projects G0H4518N and G0I2222N, and by FWO project G0B3523N (Belgium).
M.Z.~acknowledges partial support by  Methusalem grant METH/21/03 - long term structural funding of the Flemish Government.

\section{Background on $L_\infty[1]$-algebras}
In this section we recall basic notions about $L_\infty[1]$-algebras. {The latter are central objects in deformation theory, and are completely equivalent to $L_{\infty}$-algebras \cite{LadaStasheff}, which have ordinary Lie algebras and differential graded Lie algebras as special cases.} All vector spaces are assumed to be over $\mathbb K =\RR$ or $\mathbb K = \CC$.

\begin{defi}\label{def:linfty}
An \emph{$L_\infty[1]$-algebra} is a pair $(V,\{\mu_k\}_{k\geq 1})$, where
\begin{itemize}
    \item [i)] $V = \bigoplus_{i \in \ZZ} V^i$ is a $\ZZ$-graded vector space,

    \item [ii)] for every $k\geq 1$,
                \begin{equation*}
                    \mu_k\colon S^k(V) \to V
                \end{equation*}
                is a {multilinear} degree 1 map called a multibracket,
\end{itemize}
satisfying for $n\geq 0$, $x_0,\dots, x_n \in V$, 
\begin{equation}\label{eq:higherjacobi}
    \sum_{i = 0}^{n} \sum_{\sigma \in Sh(i+1,n-i-1)} \epsilon(\sigma)\mu_{n-i+1}(\mu_{i+1}(x_{\sigma(0)},\dots, x_{\sigma(i)}),x_{\sigma(i+1)},\dots,x_{\sigma(n)}) = 0.
\end{equation}
Here {$S^k(V)$ denotes the $k$-th graded symmetric power of $V$, and }  $\epsilon(\sigma)$ is the \emph{Koszul} sign, determined by
$$
x_1\dots x_n = \epsilon(\sigma)x_{\sigma(1)}\dots x_{\sigma(n)}
$$
in the graded symmetric algebra $S(V)$.
\end{defi}
 
\begin{remark}
The equations \eqref{eq:higherjacobi} are higher analogues of the Jacobi identity for Lie algebras.
\begin{itemize}
    \item [i)] For $n = 0$, it follows that $$\mu_1^2 = 0,$$ turning $(V,\mu_1)$ into a cochain complex.
    \item [ii)] For $n = 1$, it follows that $\mu_1$ is a {graded} derivation of $\mu_2$.
    \item [iii)] For $n = 2$, it follows that $\mu_3$ is a contracting homotopy of the Jacobiator of $\mu_2$, with respect to the differential $\mu_1$.
\end{itemize}
\end{remark}
Now let $(V,\{\mu_k\}_{1\leq k\leq n})$ 
be an $L_\infty[1]$-algebra with $\mu_k \equiv 0$ for $k>n$ (i.e., only finitely many multibrackets are non-zero).

\begin{defi}\label{def:MC}
A degree $0$ element  $Q \in V^0$ is a \emph{Maurer-Cartan element} if
\begin{equation}\label{eq:MCeq}
    \sum_{i = 1}^n \frac{1}{i!} \mu_i(Q,\dots,Q) = 0.
\end{equation}
\end{defi}

{A motivation for this definition is the following.}
For any  degree $0$ element $Q$, we can define new structure maps
\begin{equation}\label{eq:muQ}
\mu_k^Q := \sum_{i = 0}^\infty \frac{1}{i!} \mu_{k+i}(\underbrace{Q,\dots,Q}_{i \text{ times}},-,\dots,-),    
\end{equation}
where we note that the sum is finite. A natural question to ask is when these maps define a new $L_\infty[1]$-algebra structure on $V$. {It turns out that
if $Q$ is a Maurer-Cartan element, then  $(V,\{\mu_k^Q\}_{1\leq k \leq n})$ is an $L_\infty[1]$-algebra (see e.g. the text following  \cite[lemma 2.2.1]{FukDef}, or \cite[\S2]{DolgErr}).}

Given a degree 0 element $Q \in V^0$ and a $X\in V^{-1}$, we can construct a new degree 0 element, denoted by $Q^X$, as follows. For the following definition, we need $V^0$ to carry a topology.

\begin{defi}\label{def:gaugetransform}
Let $(V,\{\mu_k\}_{1\leq k \leq n})$ be an $L_\infty[1]$-algebra such that for every $i \in \mathbb Z$, $V^i$ carries a locally convex topology\footnote{See \cite{RudinFA} for some background on locally convex vector spaces, and see \cite{ChazarainPDEbook} for some background on function spaces.}. For $Q\in V^0$ and $X\in V^{-1}$, assume that the initial value problem
\begin{equation}\label{eq:gaugeeqgen}
\frac{d}{dt}Q_t = \mu_1^{Q_t}(X),\quad\quad Q_0 = Q
\end{equation}
has a unique solution for all $t\in [0,1]$ (notice that the right hand side was defined in equation \eqref{eq:muQ}). Then we define
$$
Q^X := Q_1 \in V^0,
$$
the value of the solution at $t=1$.
\end{defi}
\begin{remark}
When solving equation \eqref{eq:gaugeeqgen} in terms of formal paths, or when the $L_\infty[1]$-algebra is nilpotent, it can be shown that $Q^X$ is a Maurer-Cartan element if and only if $Q$ is (see for instance \cite[Corollary 1]{DotsenkoPoncin}). Since we are dealing with differentiable paths, we will need to assume {that the solution of \eqref{eq:gaugeeqgen} takes value in the space of Maurer-Cartan elements.}
\end{remark}

{Finally, we will need subspaces of $L_\infty[1]$-algebras which have an induced $L_\infty[1]$-algebra structure.}

\begin{defi}
Let $(V,\{\mu_k\}_{k\geq 1})$ be an $L_\infty[1]$-algebra, and let $W\subseteq V$ be a graded linear subspace. Then $W$ is said to be a   \emph{$L_\infty[1]$-subalgebra} if $$\mu_k(S^k(W)) \subseteq W$$ for all $k\geq 1$.
\end{defi}

\section{Main theorem for $L_\infty[1]$-algebras}

{In this subsection we present a general statement about $L_{\infty}$-algebras and their Maurer-Cartan elements. It generalizes \cite[\S 3.3]{KarandeepStability} from differential graded Lie algebras to  $L_{\infty}$-algebras.}

Assume that we have the following data: 
\medskip

\begin{mdframed}
\smallskip
\begin{itemize}
    \item [i)] An $L_\infty[1]$-algebra $(V,\{\mu_k\}_{1\leq k\leq n})$ {with finitely many non-trivial multibrackets}, such that  for each $i=-1,0,1$, $V^i$ carries a locally convex topology, 
    \item[ii)] 
    a $L_\infty[1]$-subalgebra $W$ of $V$ such that, for $i=-1,0,1$, the subspace $W^i$ is of finite codimension and closed in $V^i$,
    \item[iii)] linear splittings $\sigma_i: V^i/W^i \to V^i$ for $i = -1,0$,
    \item[iv)] a Maurer-Cartan element $Q \in W^0$,
\end{itemize}
\smallskip
such that
\smallskip
\begin{itemize}
    \item[a)] the multibrackets $\mu_k:S^k(V^0) \to V^1$ are continuous, when viewed as symmetric $\mathbb{K}$-multilinear maps,
    \item[b)] there is a convex open neighborhood $U$ of $0\in V^{-1}/W^{-1}$ such that for every $X\in U$  the following holds: the element $Q^{\sigma_{-1}(X)}$ as in definition \ref{def:gaugetransform} is  defined, the assignment
    $$
    U\times V^0 \to V^0, \quad  (X,Q') \mapsto (Q')^{\sigma_{-1}(X)}
    $$  
    is jointly continuous, and the  mod $W^0$ class {of $(Q')^{\sigma_{-1}(X)}$} depends smoothly
    on $X\in U$ for each fixed $Q'$,
    \item[c)] for $X\in U$, an element $Q' \in V^0$ is Maurer-Cartan if and only if $(Q')^{\sigma_{-1}(X)}$ is Maurer-Cartan.
\end{itemize}
\smallskip
\end{mdframed}
\medskip 
Recall that $\mu_1^Q$ was defined in eq. \eqref{eq:muQ}; we denote by $\overline{\mu_1^Q}$ the induced differential on $V/W$.
\begin{thm}\label{thm:mainthm}
Assume that we are in the setting described above. Assume that
$$
H^0(V/W,\overline{\mu_1^Q}) = 0.
$$
Then there exists an open neighborhood $\mathcal U\subset V^0$ of $Q$  such that for any Maurer-Cartan element $Q'\in \mathcal U$, there exists a family $I\subset U$, {smoothly} parametrized by an open neighborhood of 
 $$0\in \ker({\overline{\mu^Q_1}}: V^{-1}/W^{-1} \to V^{0}/W^{0}) , $$ with the property that $x \in I \implies (Q')^{\sigma_{-1}(x)}\in W^0$.
\end{thm}

In particular, $Q'$ is related (in the sense of definition \ref{def:gaugetransform}) to a Maurer-Cartan element lying in $W$.

\begin{remark}
{We provide a heuristic interpretation of the theorem. By the conclusions of the theorem, the map of moduli spaces of Maurer-Cartan elements induced by the inclusion
$$MC(W)/\sim \;\;\longrightarrow\; MC(V)/\sim$$
is surjective nearby $[Q]$.
The corresponding map of formal tangent spaces at $[Q]$ is
\begin{equation}\label{eq:formalderiv}H^0(W,\mu_1^Q) \longrightarrow H^0(V,\mu_1^Q).\end{equation}
By the hypotheses of the theorem, this linear map is surjective: indeed the obvious short exact sequence of cochain complexes gives rise to a long exact sequence in cohomology, a piece of which reads
$$\dots \longrightarrow H^0(W,\mu_1^Q)  \longrightarrow H^0(V,\mu_1^Q)\to H^0(V/W,\overline{\mu_1^Q})\longrightarrow \dots$$
While the moduli spaces of Maurer-Cartan elements are not smooth manifolds, and hence the regular value theorem can not be applied,  theorem \ref{thm:mainthm} shows that the vanishing of  $H^0(V/W,\overline{\mu_1^Q})$ is sufficient to obtain the same conclusion.} {It would be interesting to investigate under what conditions surjectivity of \eqref{eq:formalderiv} implies the conclusions of the theorem.} {This observation is analogous to the one made at the end of \cite[Remark 5.13]{Rigidity}. 
}

\end{remark}

\smallskip
{Before proving theorem \ref{thm:mainthm}, we present the main idea of the proof. The conclusion of the theorem suggests to consider, for every $Q'$ nearby $Q$, the map $$\text{ev}_{Q'}: V^{-1}/W^{-1}  \to V^0/W^0,\; v\mapsto (Q')^{\sigma_{-1}(v)} + W^0.$$ If the map $\text{ev}_{Q}$ was a submersion in a neighborhood of $0\in V^{-1}/W^{-1}$, the same would hold for its perturbation $\text{ev}_{Q'}$,  implying that its image would contain  the origin, as desired. While $\text{ev}_{Q}$ is almost never a submersion, it is transverse to a certain subspace $K\subset V^0/W^0$, therefore $\text{ev}_{Q'}$ too; from this, using the cohomological assumption
and the Maurer-Cartan condition on $Q'$, we will be able obtain the desired conclusion.}

\begin{proof}
 For simplicity, we take $U = V^{-1}/W^{-1}$, but the proof goes through for any convex open neighborhood of the origin {in $V^{-1}/W^{-1}$}.
In the first part of the proof we assume the existence of certain maps between the spaces $V^i/W^i$ for $i = -1,0,1$ with prescribed properties, from which the result follows. In the second  part we explicitly construct the maps.

{Assume the existence of the following maps:} 
\begin{itemize}
    \item[1)] A smooth map $$\text{ev}_{Q'}: V^{-1}/W^{-1}  \to V^0/W^0$$    depending continuously on $Q'\in V^0$,
    \item[2)] a   smooth map $$R_{v,Q'}:V^0/W^0\to V^1/W^1$$   depending continuously on $(v,Q') \in V^{-1}/W^{-1}\times V^0$,
\end{itemize}
with the following properties:
\begin{itemize}
    \item [A)] $\text{ev}_Q(0) = 0 \in V^0/W^0$, and {the derivative satisfies} 
    $$
    (D(\text{ev}_{Q}))_0 = \overline{\mu_1^Q}: V^{-1}/W^{-1}  \to V^0/W^0.
    $$
  Moreover, the element $(Q')^{\sigma_{-1}(v)}$  lies in the subspace $W^0\subset V^0$ if and only if $\text{ev}_{Q'}(v) = 0$.
    \item [B)] $R_{v,Q'}(0) = 0 \in V^1/W^1$ for every $(v,Q') \in V^{-1}/W^{-1}\times V^0$, and {the derivative of $R_{0,Q}$ satisfies}
    $$
    (D(R_{0,Q}))_0 = \overline{\mu_1^Q}\colon V^0/W^0 \to V^1/W^1.
    $$
    \item[C)] Whenever $Q'\in V^0$ is Maurer-Cartan, for every $v\in V^{-1}/W^{-1}$ we have:
    $$
    R_{v,Q'}(\text{ev}_{Q'}(v)) = 0.
    $$
\end{itemize}
{The following diagram summarizes diagrammatically the above maps.}
\[
\xymatrix{ V^{-1}/W^{-1} \ar[r]^{\text{ev}_{Q'}} &V^{0}/W^{0} \ar[r]^{R_{v,Q'}}&  V^{1}/W^{1}
\\
}.
\]

The conclusion of the theorem follows exactly as in \cite[theorem 3.20]{KarandeepStability}. We summarize the main ideas for the reader's convenience.
\begin{itemize}
    \item 
Let $K$ 
be a complement to $\ker(\overline{\mu_1^Q})$ 
in $V^0/W^0$. Property B) implies that $R_{0,Q}$ restricted to $K$ is an immersion at $0\in K$.  By continuity, for $(v,Q')$ close enough to $(0,Q)$, the same is true for $R_{v,Q'}$.  Therefore $R_{v,Q'}$ is injective in a neighborhood $O$ of $0\in K$.  The  neighborhood $O$  can be chosen independently of $(v,Q')$. 
\item
Property A) and the cohomological assumption imply that $\text{ev}_{Q}$ intersects $K$ transversely in $0$. Therefore, for any   $Q'$ close enough to $Q$, the map  $\text{ev}_{Q'}$ also intersects $K$ transversely, and there exists a $v\in V^{-1}/W^{-1}$ close to $0$ such that $\text{ev}_{Q'}(v) \in O$. 
\item 
When $Q'$ is Maurer-Cartan, using property C), {the fact that $R_{v,Q'}(0) = 0$ by property B), and the injectivity in the first item above,} it follows that 
$\text{ev}_{Q'}(v) = 0$. By construction, this means that  
  $(Q')^{\sigma_{-1}(v)}\in  W^0$, as desired. 
  More is true: { as $\text{ev}_{Q'}^{-1}(\{0\}) = \text{ev}_{Q'}^{-1}(O)$ is non-empty, the transversality argument above implies that $\text{ev}_{Q'}^{-1}(\{0\})$ is a submanifold of dimension equal to the one of $\ker(\overline{\mu_1^Q}:V^{-1}/W^{-1} \to V^0/W^0)$.}
\end{itemize}

We now define the maps $\text{ev}_{Q'}$ and $R_{v,Q'}$ used above.
\begin{itemize}
    \item [1)]Let $Q'\in V^0$. Then for $v \in V^{-1}/W^{-1}$, we set
$$
\text{ev}_{Q'}(v) = (Q')^{\sigma_{-1}(v)} + W^0.
$$
Then by {condition a) of the data at the beginning of this section,}
the map depends continuously on $Q'$ and smoothly on $v\in V$. 
\item[2)] Let $(v,Q') \in V^{-1}/W^{-1} \times V^0$. {To shorten the notation we write} $X:= (Q')^{\sigma_{-1}(v)}$, and $\overline{X} := X + W^0$.  
{
For $\overline{Y}  \in V^0/W^0$, we set 
\begin{align}\label{eq:Rmap2}
R_{v,Q'}(\overline{Y} ) = \sum_{i = 1}^n  \frac{1}{i!} \mu_i\Big((X- \sigma_0(\overline{X}))+\sigma_0(\overline{Y}),\dots,
(X-\sigma_0(\overline{X}))+\sigma_0(\overline{Y})\Big) + W^1.
\end{align}
}
{Since the multibrackets $\mu_i$ are continuous, the map $R$ depends continuously on the parameters $(v,Q') \in V^{-1}/W^{-1}\times V^0$.} 

{Notice that, as the $\mu_i$ are symmetric when the arguments have degree 0, we can use Newton's binomial formula to rewrite \eqref{eq:Rmap2} as 
\begin{align}\label{eq:Rmap}
R_{v,Q'}(\overline{Y}) = \sum_{i = 1}^n \sum_{j=0}^i \frac{1}{i!}\binom{i}{j} \mu_i(\underbrace{X- \sigma_0(\overline{X}),\dots, X-\sigma_0(\overline{X})}_{i-j \text{ times}},\underbrace{\sigma_0(\overline{Y}),\dots,\sigma_0(\overline{Y})}_{j \text{ times}} ) + W^1.
\end{align}
}
\end{itemize}

We check that the above maps satisfy properties {A), B), C)} above.
\begin{itemize}
\item[A)]  {The property $\text{ev}_Q(0) = 0$ holds {since $Q\in W^0$}, and 
the one} regarding the value of $\text{ev}_{Q'}$ holds by definition. For the derivative, we compute for $v\in V^{-1}/W^{-1}$ 
 \begin{align*}
    \left.\frac{d}{dt}\right|_{t=0} \text{ev}_{Q}(tv) &= \left.\frac{d}{dt}\right|_{t=0} (Q)^{t\sigma_{-1}(v)} + W^{0}\\
    &= \left. \mu_1^{Q^{t\sigma_{-1}(v)}}(\sigma_{-1}(v)) \right|_{t=0} + W^0\\
    &= {\overline{\mu_1^{Q}}(v)}.
\end{align*}
{Here in the second equality we used that 
$(Q)^{t\sigma_{-1}(v)}=\widetilde{Q}_t$, where the latter is the solution of $\frac{d}{dt}\widetilde{Q}_t = \mu_1^{\widetilde{Q}_t}(\sigma_{-1}(v))$ with initial condition $\widetilde{Q}_0 = Q$ (this is  a consequence of the fact that 
 the r.h.s. of \eqref{eq:gaugeeqgen} depends linearly on $X$ }{and the uniqueness of the solution of \eqref{eq:gaugeeqgen}.)}
\item[B)] 
{
The property $R_{v,Q'}(0) = 0 \in V^1/W^1$ 
holds since $X- \sigma_0(\overline{X})\in W^0$. 
To compute
    $(D(R_{0,Q}))_0 = \overline{\mu_1^Q}$
we notice that only the $j=1$ summand in \eqref{eq:Rmap} contributes.
    }

\item[C)] Finally, let $Q'\in V^0$, $v \in V^{-1}/W^{-1}$. As above we write $X = (Q')^{\sigma_{-1}(v)}$ and $\overline{X} = X + W^0 {=\text{ev}_{Q'}(v)}$.
 Then
{
\begin{align*}
     R_{v,Q'}(\text{ev}_{Q'}(v)) &=
     R_{v,Q'}(\overline{X}) =
      \sum_{i = 1}^n \frac{1}{i!} \mu_i(X,\dots, X) + W^1.
\end{align*}
If $Q'\in V^0$ is Maurer-Cartan, then $X$ also is (by condition c) at the beginning of this section), so the above expression vanishes.
}
\end{itemize}
 \end{proof}

\begin{remark}\label{rem:altcond}
     The continuity of the multibrackets $\mu_k: S^k(V^0) \to V^1$, {required in property a) at the beginning of this section,} was 
   {used to ensure that} 
    the map $R$ defined by
    \begin{align}\label{eq:Rpar}
        R: V^{-1}/W^{-1} \times V^0 &\to C^\infty(V^0/W^0, V^1/W^1)\\
        \nonumber (v,Q) \mapsto R_{v,Q}
    \end{align} is continuous {see item 2) of the proof of theorem \ref{thm:mainthm}).}
    Here the right hand side is equipped with the $C^1$-topology. There is however a different condition {to ensure this} which is easier to check, and which we provide in lemma \ref{lem:suffcondcont} below. 
\end{remark}
\begin{lemma}\label{lem:suffcondcont}
    Assume that there exists a closed subspace $F\subset V^0$ of finite codimension  such that the multibrackets
    $$
    \mu_k:S^k(V^0) \to V^1/W^1
    $$
    factor through $S^k(V^0/F)$. Then the map $R$ as in \eqref{eq:Rpar} is continuous.
\end{lemma}
\begin{proof}
    The proof is analogous to the proof of \cite[lemma 3.22]{KarandeepStability}, with the exception that the map $R$ now takes values in the finite-dimensional subspace consisting of \emph{polynomial} maps $V^0/F \to V^1/W^1$ of degree at most $n$.
\end{proof}

\section{Background on Dirac structures and their deformations}\label{sec:setup}
 
We recall the definition of Courant algebroids \cite{LWX}, Dirac structures \cite{Cou}, and following \cite{FZgeo} we review an $L_{\infty}[1]$-algebra governing their deformations.

\subsection{Courant algebroids}
We first need to introduce Courant algebroids \cite{LWX}.

\begin{defi}\label{ca}
A {\it Courant algebroid} over a manifold $M$ is a vector bundle $E
\to M$ equipped with a fibrewise non-degenerate symmetric bilinear
form $\la \cdot,\cdot \ra$, an $\RR$-bilinear bracket $\oo \cdot,\cdot\cc$ on
the smooth sections $\Gamma(E)$, and a bundle map $\rho: E\to TM$
called the \textit{anchor}, which satisfy the following conditions
for all $e_1,e_2,e_3\in \Gamma(E)$ and $f\in C^{\infty}(M)$:
\begin{itemize}
\item[C1)] $\oo e_1,\oo e_2,e_3\cc\cc = \oo \oo e_1,e_2\cc,e_3\cc +
\oo e_2,\oo e_1,e_3\cc\cc$,
\item[C2)] $\rho(\oo e_1,e_2\cc)=[\rho(e_1),\rho(e_2)\cc$,
\item[C3)] $\oo e_1,fe_2\cc=f\oo e_1,e_2\cc+ (\rho(e_1)f) e_2$,
\item[C4)] $\rho(e_1)\langle e_2,e_3 \rangle= \langle \oo e_1,e_2\cc,e_3 \rangle
+ \langle e_2, \oo e_1, e_3\cc\rangle$,
\item[C5)] $\oo e_1,e_1\cc = \cD \langle e_1,e_1 \rangle$.
\end{itemize}
Here we denote $\cD =\frac{1}{2}\rho^*\circ d: C^{\infty}(M)\rightarrow
\Gamma(E)$, upon 
identifying $E$
with $E^*$ using the bilinear form.
\end{defi}

 \begin{ep}[Exact Courant algebroids]\label{ep:TMT*M}
Let $M$ be a manifold, and $H$ a closed 3-form on $M$. 
The vector bundle $TM\oplus
T^*M$ acquires the structure of a Courant algebroid,  as follows \cite{SW}.  
The bilinear pairing is 
\begin{equation*}\label{pairing}
\langle X_1+\xi_1, X_2+\xi_2 \rangle = \xi_2(X_1) + \xi_1(X_2),
\end{equation*}
where $X_i+\xi_i \in \Gamma(TM\oplus TM^*)$, the anchor is the first projection, 
and the bracket is
\begin{equation}\label{Hcour}
\oo X_1+\xi_1, X_2+\xi_2\cc_H= [X_1,X_2] + \cL_{X_1} \xi_2 - i_{X_2} d\xi_1  + i_{X_2} i_{X_1} H.
\end{equation} 
We denote this Courant algebroid by $(TM\oplus T^*M)_H$
 (up to isomorphism it depends only on the cohomology class of $H$). When $H=0$ this is known as standard Courant algebroid structure.
\end{ep}

\begin{remark}\label{rem:ad}
 For any $\xi\in \Gamma(E)$, the map  $ad_{\xi}:=\oo\xi,\cdot \cc\colon \Gamma(E)\to \Gamma(E)$ is an infinitesimal automorphism of the Courant algebroid $E$. Assuming that 
 the vector field $\rho(\xi)$ is complete, $ad_{\xi}$ integrates to a 1-parameter group of automorphisms of the
 Courant algebroid $E$, which we denote by $e^{t\,ad_{\xi}}$. For instance, if $E$ is the $H$-twisted Courant algebroid as in example \ref{ep:TMT*M}, and we write $\xi=(X,\eta)\in \Gamma(TM\oplus T^*M)$, the 1-parameter group of automorphisms reads
 $e^{t\,ad_{\xi}}=(\varphi_t)_*\circ e^{B_t}$. Here $(\varphi_t)_*$
 is the tangent-cotangent lift of the flow $\varphi_t$ of $X$, and  $e^{B_t}$ is the so-called gauge transformation by the 2-form $B_t:=\int_0^t(\varphi_s)^*(d\eta-\iota_XH )\, ds$
 \cite[\S 2.2]{Hu2}\cite[\S 2.2]{Gu2}.
\end{remark}

Recall that a \emph{Lie algebroid} is a vector bundle $A\to M$ together with a Lie bracket $[\cdot,\cdot]$ on the sections $\Gamma(A)$ and a vector bundle map $\rho\colon A\to TM$ (called anchor), which are compatible in the sense that $[a_1,fa_2]=\rho(a_1)f\cdot a_2+f[a_1,a_2]$ for all sections $a_1,a_2$ and all $f\in C^{\infty}(M)$. The prototypical example is $A=TM$, and indeed Lie algebroids can be regarded as ``generalized tangent bundles''.  {Notice that at any point $p$, the bracket makes $\ker(\rho_p)$ into a Lie algebra, called isotropy Lie algebra.}
The Lie bracket and anchor of a Lie algebroid can be equivalently encoded by a degree $1$ derivation $d_A \colon \Gamma(\wedge^\bullet A^*)\to \Gamma(\wedge^{\bullet +1} A^*)$ satisfying $(d_A)^2=0$, called Lie algebroid differential, and defined by a formula analogous to the one for the de Rham differential on differential forms on a manifold.

\begin{ep}[Twisted doubles]\label{rem:double}
We generalize example \ref{ep:TMT*M}    replacing the tangent bundle $TM$ with any Lie algebroid.
Let $B$ be a Lie algebroid over $M$, and  $H\in \Gamma(\wedge^3 B^*)$ such that $d_BH=0$.
One then obtains a Courant algebroid structure on $B \oplus B^*$, with anchor given by the one of $B$ (thus vanishing on $B^*$), and with bracket given as\footnote{For this purpose, replace the Lie derivative appearing in eq. \eqref{Hcour} with Cartan's formula $\cL_{a}\xi:=\iota_a d_B\xi+d_B \iota_a\xi$ for all $a\in \Gamma(B)$ and $\xi\in \Gamma(B^*)$.} in \eqref{Hcour}. {Notice that the natural symmetric pairing on the fibers has split signature.} We denote this Courant algebroid by $(B \oplus B^*)_H$. When $H=0$, this Courant algebroid is known as the double of the Lie bialgebroid $(B,B^*)$ (where the latter is endowed with the trivial Lie algebroid structure).
\end{ep}

\subsection{Dirac structures}

\begin{defi}
Let $E\to M$ be a Courant algebroid.
  A
\emph{Dirac structure}  \cite{Cou} is a   
subbundle $A\subset E$ which is lagrangian w.r.t. the pairing (i.e. $A^{\perp}=A$), and which is involutive w.r.t. the Courant bracket.   
\end{defi}

\begin{remark}\label{Diracalgoid}
Notice that if $A\subset E$ is a Dirac structure, then the restrictions to $A$ of the anchor and of the Courant bracket make $A$ into a Lie algebroid. 
\end{remark}

\begin{ep}\label{ex:Dirac}
Following \cite{SW}, we present two classes of Dirac structure for the   Courant algebroid $(TM\oplus T^*M)_H$ of example \ref{ep:TMT*M}, where $H$ is a closed 3-form on $M$.

\begin{itemize}
    \item[a)]      Let $\omega$ be a 2-form on $M$, {and consider the associated vector bundle map $\omega^{\sharp}\colon TM\to T^*M, X\mapsto \iota_X\omega$}. Then  $\gr(\omega^{\sharp})$
 is a lagrangian subbundle of $TM\oplus T^*M$. It is a Dirac structure in $(TM\oplus T^*M)_H$ precisely when  $d\omega=-H$. 

\item[b)] Let $\pi$ be a bivector field on $M$,
 {and consider   $\pi^{\sharp}\colon T^*M\to TM, \xi\mapsto \iota_{\xi}\pi$}. 
 Then
$\gr(\pi^{\sharp})$
 is a lagrangian subbundle. {It is a Dirac structure in $(TM\oplus T^*M)_H$ precisely when $\pi$ is a \emph{$H$-twisted Poisson structure}, 
 meaning  that
\begin{equation}\label{eq:twistpoiss}
[\pi,\pi] = 2\wedge^3\pi^{\sharp}(H).
\end{equation}
}
\end{itemize}
\end{ep}

\begin{ep}\label{ex:DiracB}
Generalizing example \ref{ex:Dirac},     
{
let $B$ be a Lie algebroid over $M$, and  $H\in \Gamma(\wedge^3 B^*)$ such that $d_BH=0$.
Let $\pi\in \Gamma(\wedge^2 B)$ such that $[\pi,\pi]_B+2\wedge^3\pi^{\sharp}(H)
=0$. Then $\text{graph}(\pi)$ is a Dirac structure in  the Courant algebroid  $(B \oplus B^*)_H$ defined in remark \ref{rem:double}. 
}

In particular, $0\oplus B^*$ is a Dirac structure. 
When the twist $H$ is not exact, then this Dirac structure does not admit any Dirac complement. Indeed, any such complement would be the graph of an element $\Omega\in \Gamma(\wedge^2B^*)$ satisfying $-d_B\Omega=H$, yielding a contradiction. (In particular, $B \oplus 0$ is not a Dirac structure). 
 \end{ep}
 
 \begin{remark}\label{rem:untwist}
 Notice that for any $\omega \in \Gamma(\wedge^2 B^\ast)$, there is an isomorphism of Courant algebroids given by
 $$
\exp(\omega^{{\sharp}}):(B\oplus B^\ast)_H \to (B\oplus B^\ast)_{H-d\omega}, \quad \exp(\omega^{{\sharp}})(X+\alpha) = X + \alpha + \iota_X\omega.
$$
This was observed in \cite[\S 2.2]{Gu2} for $B = TM$. In particular, when $H$ is exact, we have $(B\oplus B^\ast)_H \cong (B\oplus B^\ast)_0$.
\end{remark}

{
\begin{remark}
A Courant algebroid over a manifold $M$ induces a partition of $M$ into immersed submanifolds of varying dimension (called leaves) which are tangent to the image of the anchor map. The same applies for Dirac structures.   
\end{remark}
}

\subsection{Deformations of Dirac structures}\label{subsec:defDirac}
Let $E\to M$ be a Courant algebroid, and let $A\subset E$ be a Dirac structure. In other to give a description of the Dirac structures nearby $A$, we make an auxiliary choice of lagrangian complement $K$ (so $E=A\oplus K$ as  vector bundles), and express the Courant algebroid structure in terms of $A$ and $K$.

\begin{remark}\label{rem:split}
A lagrangian complement of $A$ exists if and only if $E$ is of even rank $2n$ and the pairing has signature $(n,n)$
(see e.g. \cite[corollary 4.4]{kwdirac}).
\end{remark}

Since the bracket
$[\cdot,\cdot]_A:=  \oo \cdot,\cdot \cc|_A$  and the bundle map $\rho|_{A}\colon A\to TM$ make $A$ into
a Lie algebroid (see remark \ref{Diracalgoid}),
we denote by $d_A$ the corresponding  Lie algebroid differential (it squares to zero).

Identify  $K\cong A^*$ via the pairing on the fibers of $E$, i.e. via $K\overset{\simeq}{\longrightarrow} A^*,\ u\longmapsto\langle u,\,\cdot\, \rangle|_A$. Notice that 
$A^*$ is usually not a Dirac structure.
Similarly to the above, the restriction
$[\eta_1,\eta_2]_{A^*}:=pr_{A^*}(\oo(0,\eta_1),(0,\eta_2)\cc)$ on $\Gamma(A^*)$ and the bundle map $\rho|_{A^*}\colon A^*\to TM$ allow one to write
down a degree $1$ derivation $d_{A^*}$  
of $\Gamma(\wedge^\bullet A)$, which generally does not square to zero.

Consider also the map $$\Gamma(\wedge^2 A^*)\to \Gamma(A)\;,\; \eta_1 \wedge \eta_2 \mapsto pr_{A}(\oo(0,\eta_1),(0,\eta_2)\cc),$$
which measures the failure of $A^*$ to be a Dirac structure, 
and view it as an   element  $\Psi\in \Gamma(\wedge^3 A).$ 

 From $\Psi$,   $(A,[\cdot,\cdot]_A,\rho|_A)$, and    $(A^*,[\cdot,\cdot]_{A^*},\rho|_{A^*})$   one can reconstruct the Courant algebroid structure on $E=A \oplus A^*$:   the bracket is recovered  as
\begin{align}\label{CB}
&\oo (a_1,\eta_1),(a_2,\eta_2) \cc=\\
&\Big( [a_1,a_2]_A+\cL_{\eta_1}a_2-\iota_{\eta_2}d_{A^*}a_1+\Psi(\eta_1,\eta_2,\cdot)\;,\;[\eta_1,\eta_2]_{A^*}+\cL_{a_1}\eta_2-\iota_{a_2}d_{A}\eta_1\Big)\nonumber
\end{align}
and the anchor as $\rho_A+\rho_{A^*} \colon A\oplus A^* \to TM$  
{(\cite[\S 3.8]{DimaThesis}, see also \cite[\S 3.2]{YvetteAllThat}).}

The statement of \cite[lemma 2.6]{FZgeo} reads as follows\footnote{The global minus in front of the ternary bracket, was erroneously omitted in \cite[lemma 2.6]{FZgeo}.}:
\begin{prop}\label{prop:linfty}
The graded vector space $\Gamma(\wedge^\bullet A^*)[2]$ has an   $L_\infty[1]$-algebra structure\footnote{This $L_\infty[1]$-algebra structure depends on the choice of $K$, but it is independent of this choice up to $L_\infty[1]$-isomorphism \cite{GMS}\cite{TORTORELLADefDiracJacobi}.}
  $\{\mu_k\}$, whose only non-trivial multibrackets $\mu_1,\ \mu_2,\ \mu_3$ are defined as follows:
		\begin{align*}
		\mu_1(\alpha{[2]})&=(d_A\alpha){[2]}\\
		\mu_2(\alpha{[2]},\beta{[2]})&=(-1)^{|\alpha|}[\alpha,\beta]_{A^*}{[2]}\\
		\mu_3(\alpha{[2]}, \beta{[2]}, \gamma{[2]})&=-(-1)^{|\beta|} (\alpha^\sharp\wedge \beta^\sharp \wedge \gamma^\sharp) \Psi{[2]}.
		\end{align*}

Further,  MC elements $\pp \in \Gamma(\wedge^2 A^*)$ of this $L_\infty[1]$-algebra 
parametrize Dirac structures $L\subset E$ that are transverse to $K$, via
	\begin{equation*}
	L=\gr(\pp^{\sharp})=\{a{+}\iota_a\pp^{\sharp}\mid \xi\in A\}\subset A\oplus A^\ast\cong E.
	\end{equation*}
\end{prop}

Here we define  $\alpha^{\sharp}a:=\iota_a \alpha$, and  
$$
(\alpha^\sharp\wedge\beta^\sharp\wedge \gamma^\sharp)(x_1\wedge x_2\wedge x_3)=\sum_{\sigma\in S_3}(-1)^\sigma \alpha^\sharp(x_{\sigma(1)})\wedge \beta^\sharp(x_{\sigma(2)})\wedge \gamma^\sharp(x_{\sigma(3)}),
$$
 for all homogeneous $\alpha,\beta,\gamma \in \Gamma(\wedge^\bullet A^*)$ and all $x_i\in \Gamma(A)$.

\section{Gauge equivalences for Dirac structures}
As in  \S \ref{subsec:defDirac}, let $E$ be a Courant algebroid, $A$ a Dirac structure, and choose a lagrangian complement, which we identify with $A^*$ using the pairing (hence $E=A\oplus A^*$ as vector bundles). The main result of this section is proposition \ref{prop:CAauto},  which gives a geometric description of the gauge equivalence relation that the $L_{\infty}[1]$-algebra of proposition \ref{prop:linfty} induces on the Dirac structures nearby $A$.

Recall from remark \ref{rem:ad} that any element $\xi\in \Gamma(A^*)$ induces a one-parameter group of Courant algebroid automorphisms defined for small $t$, via $e^{t\,ad_\xi}$, where $ad_\xi=\oo \xi, \cdot \cc$. We will use repeatedly the following fact, which follows immediately from eq. \eqref{CB}: 
$$ad_\xi a=(\cL_{\xi}a, -\iota_{a }d_{A}\xi)$$
for all $a\in \Gamma(A)$.

Let $\pp\in\Gamma(\wedge^2 A^*)$ be a Maurer-Cartan element of the $L_{\infty}[1]$-algebra   $\Gamma(\wedge^\bullet A^*)[2]$
of proposition \ref{prop:linfty} (hence $\gr(\pp^{\sharp})$
 is a Dirac structure).
For any compactly supported $\xi\in \Gamma(A^*)$, we obtain a smooth one-parameter family of 
Maurer-Cartan elements, given by the {unique} solution $\pp_t$ of the equation
\begin{equation}\label{eq:gaugeeq}
\dot{\pp}_t=-d_A\xi+[\xi,\pp_t]_{A^*}+\frac{1}{2}(\xi^{\sharp}\wedge \pp_t^{\sharp}\wedge \pp_t^{\sharp})\Psi,
\end{equation}
subject to the initial condition $\pp_0=\pp$.
This is the gauge equation associated to the element $-\xi$ in the    $L_\infty[1]$-algebra $\Gamma(\wedge^\bullet A^*)[2]$, cf. equation \eqref{eq:gaugeeqgen}.

 \begin{remark}
 {We call \emph{gauge equivalence relation} the equivalence relation on Maurer-Cartan elements  generated by the following: two Maurer-Cartan elements are related if they can be written as 
 $\pp_0$ and $\pp_1$ as above for some $\xi\in \Gamma(A^*)$. For a comparison of the gauge equivalence relation with other notions in terms of polynomial paths
 found in the literature, see \cite[proposition 9]{DotsenkoPoncin} (see also \cite[remark 5.22]{KraftSchnitzerIntro}).
 }
 \end{remark}
 
{The following proposition states that the 1-parameter family of Dirac structures  $\gr(\pp^{\sharp}_t)$ is obtained applying Courant algebroid automorphisms to $\gr(\pp^{\sharp})$.}
 
\begin{prop}\label{prop:CAauto} Let $E=A\oplus A^*$ be a Courant algebroid as in \S \ref{subsec:defDirac}. 
Let  $\xi\in \Gamma(A^*)$ be compactly supported, and $\pp\in\Gamma(\wedge^2 A^*)$ be a Maurer-Cartan element {of the $L_{\infty}[1]$-algebra of proposition \ref{prop:linfty}}. 
Let $\pp_t\in\Gamma(\wedge^2 A^*)$ be determined by the property
\begin{equation}
\label{eq:graph}
\gr(\pp^{\sharp}_t)=e^{t\,ad_{\xi}}\gr(\pp^{\sharp}),
\end{equation}
for   $t\in \RR$ close enough to zero.

Then $\pp_t$ is the unique solution of eq. \eqref {eq:gaugeeq} satisfying $\pp_0=\pp$. 
\end{prop}

\begin{remark}\label{rem:iso}
Since $pr_A\colon \gr(\pp^{\sharp})\to A$ is an isomorphism, by continuity we have that
$pr_A\colon e^{t\,ad_{\xi}}\gr(\pp^{\sharp})\to A$ is an isomorphism for  $t$ in an open interval around zero, since $\xi$ is compactly supported.
\end{remark}

\begin{proof}
Given $a\in A$, we use the notation $$Y^a_t:=e^{t\,ad_{\xi}}(a+\pp^{\sharp}a).$$ Then
the R.H.S. of eq. \eqref{eq:graph} can be written as
$\{Y^a_t: a\in A\}$.
So eq.   \eqref{eq:graph} is equivalent to the condition that 
\begin{equation}\label{eq:grapht}
\pp_t^{\sharp}(pr_A(Y^a_t))=pr_{A^*}(Y^a_t) 
\end{equation}
for all $a\in A$ (here we made use of remark \ref{rem:iso}).
 
Now adopt the notation $$x^a_t:=pr_A( Y^a_t).$$
Notice that 
\begin{equation}\label{eq:advY}
ad_{\xi} Y^a_t=ad_{\xi}\left(x^a_t+pr_{A^*}( Y^a_t)\right)=ad_{\xi}(x^a_t)+ad_{\xi}(\pp_t^{\sharp}x^a_t)
\end{equation}
 using eq. \eqref{eq:grapht} in the last equality.

For every section $a\in \Gamma(A)$, we take the time derivative of eq. \eqref{eq:grapht}, and write it out using
 eq. \eqref{CB} and \eqref{eq:advY}:
\begin{itemize}
\item taking the time derivative of the LHS we get 
$$\dot{\pp}_t^{\sharp}(pr_A( Y^a_t))+\pp_t^{\sharp}(pr_A(ad_{\xi}  Y^a_t))=
\dot{\pp}_t^{\sharp}(x^a_t)+\pp_t^{\sharp}\left((\cL_\xi x^a_t)+ \Psi(\xi,\pp_t^{\sharp} x^a_t,\,\cdot\,)\right).$$
\item
Taking the time derivative of the RHS  of eq. \eqref{eq:grapht}, we get  
 $$pr_{A^*}(ad_{\xi}  Y^a_t)=-\iota_{x^a_t}d_A\xi+[\xi, \pp_t^{\sharp} x^a_t]_{A^*}.$$
\end{itemize}

 Hence the time derivatives of the LHS and RHS of  eq. \eqref{eq:grapht} are the same if{f}
 \begin{equation}\label{eq:ddt}
 \dot{\pp}_t^{\sharp}(x^a_t)=-\iota_{x^a_t}d_A\xi+[\xi, \pp_t^{\sharp} x^a_t]_{A^*}-\pp_t^{\sharp}(\cL_\xi x^a_t)- \pp_t^{\sharp}(\Psi(\xi,\pp_t^{\sharp} x^a_t,\,\cdot\,)).
\end{equation}

 Using lemma \ref{lem:idLA} and lemma \ref{lem:cubic} below, we see that the RHS of eq. \eqref{eq:ddt} can be written as  
 $$\iota_{x^a_t}\left(-d_A\xi+[\xi,\pp_t]_{A^*}+\frac{1}{2}(\xi^{\sharp}\wedge \pp_t^{\sharp}\wedge \pp_t^{\sharp})\Psi\right).$$
 This, together with remark \ref{rem:iso}, shows that 
  $\pp_t$  is the unique solution of  the differential equation \eqref{eq:gaugeeq} with $\pp_0=\pp$.
\end{proof}

\begin{lemma}\label{lem:idLA}
For all $\xi\in \Gamma(A^*), \pp\in \Gamma(\wedge^2A^*), a\in \Gamma(A)$ the following identity holds:
$$\iota_a[\xi,\pp]_{A^*}=[\xi,\pp^{\sharp} a]_{A^*}-\pp^{\sharp}(\cL_\xi a).$$
\end{lemma}
\begin{proof}  
Let $\Theta$ denote the degree $3$ function  on $T^*[2]A[1]$ that, together with the degree $-2$ Poisson bracket of ``functions'' $\{\cdot,\cdot\}$, encodes the Courant algebroid structure of $A\oplus A^*$ (see \cite[\S 2.2]{FZgeo} and references therein). 

\emph{Claim:} $\{ \{\Theta, \xi\},\pp\}$ equals $[\xi,\pp]_{A^*}\in \Gamma(\wedge^2A^*)$ plus an element of $\Gamma(A^*\otimes A)$.

To prove the claim, we may assume that $\pp=\eta_1\wedge \eta_2$ for $\eta_i\in \Gamma(A^*)$. 
Notice that by definition $\{ \{\Theta, \xi\},\eta_1\}=\oo \xi, \eta_1 \cc$ equals $[x,\eta_1]_{A^*}$ plus an element of $\Gamma(A)$.
The claim follows from applying the Leibniz rule to 
$\{ \{\Theta, \xi\},\eta_1\cdot \eta_2\}$.
 
From the claim it follows that for all $b\in \Gamma(A)$,
\begin{equation}\label{eq:super}
\iota_b\iota_a[\xi,\pp]_{A^*}=\{b,\{a,\{ \{\Theta, \xi\},\pp\}\}\}.
\end{equation}
Now the graded Jacobi identity for $\{\cdot,\cdot\}$ implies 
\begin{align*}
\{a,\{ \{\Theta, \xi\},\pp\}\}= -\{ \{\Theta, \xi\},\{\pp,a\}\}-\{\{ \{\Theta, \xi\},a\},\pp\}=\oo \xi,\iota_a\pp \cc-\pp^{\sharp}(\cL_\xi a),
\end{align*}
where to compute the last term we used that the restriction of $\{\cdot,\cdot\}$
to $\Gamma(A^*\otimes A)$ is the pairing, that
 $A^*$ is isotropic and $pr_A \oo \xi,a \cc=\cL_\xi a$. It follows that the R.H.S. of 
 eq. \eqref{eq:super} equals $\iota_b\left([\xi,\iota_a\pp]_{A^*}-\pp^{\sharp}(\cL_\xi a)\right)$.
  \end{proof}

\begin{lemma}\label{lem:cubic}
For all $\xi\in \Gamma(A^*), \pp\in \Gamma(\wedge^2A^*), a\in \Gamma(A)$ and $\Psi\in \Gamma(\wedge^3 A)$ the following identity holds:
$$-\pp^{\sharp}\left(\Psi(\xi,\pp^{\sharp} a,\,\cdot\,)\right)=
\frac{1}{2}\iota_a\left((\xi^{\sharp}\wedge \pp^{\sharp}\wedge \pp^{\sharp})\Psi\right)$$
\end{lemma}
\begin{proof}
The L.H.S. equals $\Psi(\xi,\pp^{\sharp} a,\pp^{\sharp}\cdot)$. For the R.H.S., we may assume that $\Psi$ is decomposible, i.e. $\Psi=x_1\wedge x_2\wedge x_3$ for $x_i\in \Gamma(A)$. We then compute $$(\xi^{\sharp}\wedge \pp^{\sharp}\wedge \pp^{\sharp})\Psi=2 \xi^{\sharp}(x_1)  \cdot \pp^{\sharp}x_2\wedge  \pp^{\sharp}x_3 +{cycl.\; perm.},$$
and using the relation $\langle \pp^{\sharp}x_2, a \rangle=-\langle x_2, \pp^{\sharp} a \rangle$ it follows that
\begin{equation*} \iota_a\left((\xi^{\sharp}\wedge \pp^{\sharp}\wedge \pp^{\sharp})\Psi\right)=
2\Psi(\xi,\pp^{\sharp} a,\pp^{\sharp}\,\cdot\,).
\end{equation*}	
\end{proof}
\begin{remark}
{The statement of proposition \ref{prop:CAauto} admits a version in which $\xi\in \Gamma(A^*)$ is replaced by a smooth 1-parameter family of elements of $\Gamma(A^*)$,   providing
a unified approach to the geometric characterization of 
the equivalences of various kinds of geometric structures (e.g. the foliations and pre-symplectic structures worked out in \cite{SZpreequi}).}
\end{remark}

\section{An application: Stability of fixed points of Dirac structures}\label{sec:stabDirac}

{
We start with a definition:
\begin{defi}
Consider a Courant algebroid $E$ with anchor map $\rho\colon E\to TM$ and a Dirac structure $A$. We say that a point $p\in M$ is a \emph{fixed point of $A$} whenever $\rho(A_p)=0$.
\end{defi}
\noindent In this section we obtain a stability criterion for fixed points of Dirac structures. We do so applying suitably theorem \ref{thm:mainthm}; this yields proposition \ref{thm:stableDirac}, which we then express in more geometric and explicit terms as theorem \ref{thm:stableDirac2}.
}

\subsection{Applying theorem \ref{thm:mainthm}}

Assume that we are in the setup of \S \ref{subsec:defDirac}: 
 $E\to M$ is a Courant algebroid,  $A\subset E$   a Dirac structure. Make an auxiliary choice of lagrangian complement to $A$, identify the complement with $A^*$ via the pairing, and denote by $\rho_{A^\ast}$ the restriction of the anchor of $E$ to $A^*$.
Suppose that $p \in M$ is a fixed point of the Dirac structure $A$, i.e. {$\rho(A_p)=0$}. {We want to apply theorem \ref{thm:mainthm} to the following data:}

\begin{itemize}
    \item[i)]
    We take the  $L_\infty[1]$-algebra $V = \Gamma(\wedge^\bullet A^\ast)[2]$, with the brackets $\mu_1,\mu_2,\mu_3$ as in proposition \ref{prop:linfty}. We equip
    $V^{-1}$ with the $C^\infty$-topology, $V^0$ with the $C^1$-topology and $V^1$ with the $C^0$-topology.
    \item[ii)] 
    We take the   $L_\infty[1]$-subalgebra $W\subset V$
    defined by 
    \begin{equation}\label{eq:dirsubalg}
    W^i:= \{\Lambda \in \Gamma(\wedge^{i+2} A^\ast) \mid \Lambda_p \in \wedge^{i+2} \ker((\rho_{A^\ast})_p)\}.
    \end{equation}
    In lemma \ref{lem:subalg}  we check that $W$ indeed is an $L_{\infty}[1]$-subalgebra and that $W^i\subset V^i$ is closed {for $i=-1,0,1$}. Notice that the Maurer-Cartan elements of $W$ are precisely those $\pp \in \Gamma(\wedge^2 A^\ast)$, such that $\gr(\pp^{\sharp})$ is Dirac, and $p\in M$ is a fixed point of $\gr(\pp^{\sharp})$.
    
    \item[iii)] {For $i=-1,0$ we pick splittings $\sigma_{i}: V^i/W^i\to V^i=\Gamma(\wedge^{i+2} A^\ast)$ consisting of compactly supported sections. This is possible since  $V^i/W^i$ are finite-dimensional vector spaces.}
    \item[iv)] As Maurer-Cartan element in $W^0$ we pick $Q = 0$. Notice that we are considering  {Maurer-Cartan elements near $0$, which correspond to Dirac structures near $A$.}
    \end{itemize}
    These choices satisfy the properties  required {just before theorem \ref{thm:mainthm}}:
    \begin{itemize}
     \item[a)]
    {As pointed out in remark \ref{rem:altcond}, the continuity of the multibrackets was required in order to make the map $R$ in equation \eqref{eq:Rpar} continuous. However, lemma \ref{lem:suffcondcont} provides an alternative condition for $R$ to be continuous. We therefore instead check that the conditions of lemma \ref{lem:suffcondcont} are satisfied. }\\
    Note that the values of $\mu_1$ and $\mu_2$ in a point $q\in M$ only depend on the first jet of the arguments in $q$. Moreover, the value of $\mu_3$ in a point only depends on the values of the arguments in $q$. Consequently, $F = I_p^2\Gamma(\wedge^2 A)$ satisfies the assumptions of lemma \ref{lem:suffcondcont}.
    \item[b)] By the choice of lifts to compactly supported sections {in iii) above}, the gauge action exists for all $t\in \RR$: {for any $\xi \in \Gamma(A^*)$ the action $e^{t\,ad_{\xi}}$}   is defined as long as the flow of $\rho(\xi)\in \mathfrak X(M)$ is defined. {As $U$} we can therefore take any open neighborhood of the origin {in $V^{-1}/W^{-1}$}. The continuity and smoothness assertions of the gauge action follow from a standard argument using the smoothness of $\pp \in \Gamma(\wedge^2 A^\ast)$ and the fact that the topology on the $V^i$ is defined by uniform convergence of some jet of $\pp$ on compact sets.
    \item[c)] Note that an element $\pp \in V^0$ is Maurer-Cartan if and only if its graph is involutive, by proposition \ref{prop:linfty}. As the gauge action is by Courant algebroid automorphisms ({see proposition \ref{prop:CAauto}}), involutivity is preserved.
\end{itemize}

We apply theorem \ref{thm:mainthm} to the data i)-iv) above. {In doing so we   invoke proposition \ref{prop:CAauto}, and we use that $\mu_1=d_A$.
 We further use the isomorphism of chain complexes $(V/W, \overline{\mu_1})\cong 
 \left(\frac{\wedge^{\bullet} A^\ast_p}{\wedge^{\bullet}\ker((\rho_{A^\ast})_p)}[2],\overline{d_A}\right)$ given by evaluation at $p$ (hence the differential $\overline{d_A}$ is computed extending to an element of $\Gamma(\wedge^{\bullet} A^\ast)$, applying $d_A$ and evaluating at $p$).}
We obtain:

\begin{prop}
\label{thm:stableDirac}
{Let $E\to M$ be a Courant algebroid with anchor $\rho$, and let $A\subset E$ be a Dirac structure. Choose a  lagrangian complement, which we canonically identify with $A^*$  via the pairing.} 
Let $p\in M$ be a fixed point of the Dirac structure $A$, i.e.  $\rho(A_p)=0$. 
Assume that
$$
H^2\left(\frac{\wedge^{\bullet} A^\ast_p}{\wedge^{\bullet}\ker((\rho_{A^\ast})_p)},\overline{d_A}\right) = 0.
$$

Then there exists an $C^1$-open neighborhood $\mathcal U$ of $0 \in \Gamma(\wedge^2 A^\ast)$, such that for any $\pp \in \mathcal{U}$ for which $\gr(\pp^{\sharp})$ is Dirac, the following holds:
there is a {smooth} family $I \subset A^\ast_p/\ker((\rho_{A^\ast})_p)$, parametrized by a neighborhood of $$ 0 \in \ker\left({\overline{d_A}}: \frac{A^\ast_p}{\ker((\rho_{A^\ast})_p)} \to \frac{\wedge^2 A^\ast_p}{\wedge^2 \ker((\rho_{A^\ast})_p)}\right),$$ 
with the property that $x \in I$ implies:
\begin{equation}\label{eq:fixed}
    p\text{ is a fixed point of } 
{e^{ad_{\sigma_{-1}(x)}}(\gr(\pp^{\sharp}))}.
\end{equation}
{Here $\sigma_{-1}\colon \frac{A^\ast_p}{\ker((\rho_{A^\ast})_p)} \to \Gamma(A^*)$ is a fixed splitting taking values in compactly supported sections.}
\end{prop}

\begin{remark}\label{rem:stabplus}
Instead of $W$ defined by equation \eqref{eq:dirsubalg}, one could take $\widetilde{W}^\bullet := I_p\Gamma(\wedge^{\bullet+2} A^\ast)$. Notice that Maurer-Cartan elements in $\widetilde{W}$ correspond to Dirac structures which coincide with $A$ at $p$. Applying theorem \ref{thm:mainthm} to this data, {one obtains the following statement: assume the vanishing of $$H^2(\Gamma(\wedge^{\bullet} A^\ast)/I_p\Gamma(\wedge^\bullet A^*)) \cong H^2(\wedge^\bullet A^*_p,\overline{d_{A}}).$$} Then for any Dirac structure $L$ near $A$, there is a family $I\subset A_p^\ast$, parametrized by a neighborhood of $$
0 \in \ker(\overline{  d_A}:A_p^\ast \to \wedge^2 A_p^\ast),
$$
with the property that $x \in I$ implies: 
$$
(e^{ad_{\sigma_{-1}(x)}}(L))_p = A_p.
$$
{From this it follows} that $p$ is a fixed point of $e^{ad_{\sigma_{-1}(x)}}(L)$, but in general the converse does not hold {(see remark \ref{rem:stablesection} later on)}.
\end{remark}

 \subsection{A geometric restatement}

In this subsection we rephrase the hypotheses and the conclusions  of proposition  \ref{thm:stableDirac}, obtaining in theorem \ref{thm:stableDirac2} a  geometric statement which does not make reference to any choice of lagrangian complements.

Expressing
  the conclusion of proposition \ref{thm:stableDirac} in terms of $\gr(\pp^{\sharp})$, we see that  $\gr(\pp^{\sharp})$ has a fixed point nearby $p$:
\begin{lemma}\label{cor:fixedpt}
{Let $\pp \in  \Gamma(\wedge^2 A^\ast)$ and $x\in I$  be as in proposition \ref{thm:stableDirac}.} Then \eqref{eq:fixed} holds if and only if:
$$\text{$\phi_{-\rho(\sigma_{-1}(x))}(p)$ is a fixed point of $\pp$}.$$
Here $\phi_{-\rho(\sigma_{-1}(x))}$ denotes the time-$1$ flow of the vector field $-\rho(\sigma_{-1}(x))$ on $M$.
\end{lemma}
\begin{proof}
Write $\xi := \sigma_{-1}(x)$.
Since $e^{-ad_{\xi}}$ is a Courant algebroid automorphism, we have
$
\rho \circ e^{-ad_\xi} = (\phi_{-\rho(\xi)})_\ast \circ \rho$.
where $\phi_{-\rho(\xi)}$ is the time one flow of the vector field $-\rho(\xi)$. Hence for all $Y\in E_p$ we have
$$\text{$\rho_{\phi_{-\rho(\xi)}(p)}(e^{-ad_\xi} Y)=0$ if{f} $\rho_p(Y)=0$.}$$ 
Applying this to all $Y\in e^{ad_{\xi}}\gr(\pp^{\sharp})$ the conclusion follows.
\end{proof}

{We address how the fixed points of lemma \ref{cor:fixedpt} depend on the parameters.}
\begin{lemma}\label{lem:diffeo}
The map 
$$\Phi\colon A^*_p/\ker(\rho_{A^\ast_p})\to N,\;\; y\mapsto \phi_{-\rho(\sigma_{-1}(y))}(p),$$
restricted to a suitable neighborhood of the origin,  
is  a  diffeomorphism onto its image.
Here  $N$ is the leaf of the Courant algebroid through $p$.
\end{lemma}
\begin{proof}
    The anchor at $p$ induces a linear isomorphism \begin{equation}\label{eq:isoleaf}
 A^*_p/\ker(\rho_{A^\ast_p})\to \rho(A^*_p)=\rho(E_p)
    \end{equation}
onto $T_pN$.
(The equality holds since $\rho(A_p)=\{0\}$).
In terms of the splitting $\sigma_{-1}$, the above isomorphism is $y\mapsto \rho_{A^\ast}(\sigma_{-1}(y))|_p$.  
{Composing first with $-Id_{T_pN}$  and then with the map $\Psi$ obtained applying lemma \ref{lem:spray}, we obtain exactly $\Phi$.} 
\end{proof}

We now rephrase the hypotheses of proposition \ref{thm:stableDirac}, without making reference to the choice of lagrangian complement.

\begin{lemma}\label{lem:hypo}
Consider the Lie algebra $\g:=A_p$, and  denote by   $d_{\g}$
its Chevalley-Eilenberg differential.

i) The subspace \begin{equation}\label{eq:hideal}
\mathfrak{h}:=(\ker(\rho|_{E_p}))^{\perp}
\end{equation}
is a Lie ideal of $\g$. {Here the orthogonal is taken w.r.t. the symmetric pairing.}

ii) The cochain complex appearing in proposition \ref{thm:stableDirac} agrees with
$\left(\frac{\wedge^{\bullet} \g^*}{\wedge^{\bullet}\h^{\circ}},\overline{d_{\g}}\right)$. In particular, it is independent of the choice of lagrangian complement to $A$.
\end{lemma}
 
 \begin{proof}
 {We first motivate the definition of $\h$.}
 For any lagrangian complement $K$ of $A$, recall that we make use of the identification $K\cong A^*, k\mapsto \langle k, \,\cdot\, \rangle|_A$.
{Using  that $A_p$ is lagrangian and that the anchor vanishes on  $A_p$, one can see that under this identification, $\ker(\rho|_{K_p})$ is mapped to 
\begin{equation}\label{eq:kerA*}
 \{\langle e, \,\cdot\, \rangle|_{A_p}  : e\in   \ker(\rho|_{E_p})\}=\h^{\circ}.
\end{equation}
The equality holds because the annihilator of the l.h.s of \eqref{eq:kerA*} is given by $\ker(\rho|_{E_p}))^{\perp}\cap A_p$, which agrees with $\h$.  
}

i) To see that $\h$ is a Lie ideal, we need to check that for any $h\in \h$ and $a\in A_p$ we have $[a,h]\in (\ker(\rho|_{E_p}))^{\perp}$. Take $e\in \ker(\rho|_{E_p})$. Extending $a,h$ (respectively $e$) to sections of $A$ (respectively $E$), we have
 $$\langle e, \oo a,h \cc \rangle=-\langle \oo a,e\cc, h \rangle
 +\rho(a)\langle e, h \rangle$$
by property   C4) in definition 
\ref{ca}. 
 This function vanishes at $p$, since $\ker(\rho|_{E_p})$ is closed under the Courant bracket and since $\rho(a)$ is a vector field vanishing at $p$.

 ii) The Chevalley-Eilenberg differential $d_{\g}$ 
 preserves $\wedge^\bullet \h^{\circ}$, since $\h$ is a Lie ideal (it suffices to check this for elements of $\h^{\circ}$).  Hence $d_{\g}$ descends to a differential on the quotient $\frac{\wedge^{\bullet} \g^*}{\wedge^{\bullet}\h^{\circ}}$.
Observe that  {the quotient map}
\begin{equation}\label{eq:resmap}
(\wedge^\bullet A^\ast_p, \overline{d_A})\to \left(\frac{\wedge^\bullet \g^*}{\wedge^\bullet \mathfrak h^\circ}, \overline{d_{\g}}\right)
\end{equation}
is a surjective chain map { with kernel given by $\wedge^\bullet \h^\circ$.
Since $\mathfrak h^\circ = \langle\ker(\rho|_{K_p}),\,\cdot\,\rangle|_{A_p}$ 
{(see equation \eqref{eq:kerA*})},  
the  map \eqref{eq:resmap} descends to an isomorphism between
$\left(\frac{\wedge^{\bullet} A^\ast_p}{\wedge^{\bullet}\langle\ker(\rho|_{K_p}),\,\cdot\,\rangle|_{A_p}},\overline{d_A}\right)$ -- as defined just before proposition \ref{thm:stableDirac} -- 
and
$\left(\frac{\wedge^{\bullet} \g^*}{\wedge^{\bullet}\h^{\circ}},\overline{d_{\g}}\right)$}.
\end{proof}

We finally can rephrase proposition  \ref{thm:stableDirac} in a more geometric way, and without making reference to lagrangian complements.

\begin{thm}[Stability of fixed points of Dirac structures]\label{thm:stableDirac2}
{Let $E\to M$ be a Courant algebroid whose pairing has split signature.
Let $A\subset E$ be a Dirac structure which has a fixed point at $p\in M$, i.e.  $\rho(A_p)=0$. Denote by $\g$ the Lie algebra $A_p$, and consider its Lie ideal $\mathfrak{h}:=(\ker(\rho|_{E_p}))^{\perp}$.}
Assume that
$$
H^2\left(\frac{\wedge^{\bullet} \g^*}{\wedge^{\bullet}\h^{\circ}},\overline{d_{\g}}\right) = 0.
$$
Fix a neighborhood $\widetilde{N}$ of $p$ inside the corresponding leaf of the Courant algebroid.

Then there exists an $C^1$-open neighborhood $\mathcal U$ of 
$A$ in the space of Dirac structures,
such that for any $L \in \mathcal{U}$ 
there is a  submanifold $F^L$ of $\widetilde{N}$ consisting of fixed points of $L$. The dimension of $F^L$ equals that of
$\ker\left(\overline{d_{\g}}\colon \frac{ \g^*}{ \h^{\circ}}\to \frac{\wedge^2 \g^*}{\wedge^2\h^{\circ}}\right).$
 \end{thm}

\begin{proof}
Since $E$ has split signature, there exists a lagrangian complement to $A$, see remark \ref{rem:split}. This allows us to apply proposition \ref{thm:stableDirac}. We do so making the following choices  {in iii) and b)} at the beginning of \S \ref{sec:stabDirac}: the splitting 
$\sigma_{-1}\colon \frac{A^\ast_p}{\ker((\rho_{A^\ast})_p)} \to \Gamma(A^*)$   takes values in sections which, once restricted to the leaf, are supported in $\widetilde{N}$;
 the open neighborhood $U$ of the origin in the domain 
 {is such that $\Phi|_U$ is a diffeomorphism onto its image, where $\Phi$ is the map of lemma \ref{lem:diffeo}.}

The cohomological obstruction that appears in proposition \ref{thm:stableDirac} is identical to  the one of the present theorem, by lemma \ref{lem:hypo}.

The conclusions of proposition \ref{thm:stableDirac} imply those of the present theorem. To see this, notice that any Dirac structure $L$ close enough to $A$ is the graph of some element of $\Gamma(\wedge^2 A^*)$.
Consider the submanifold $I\subset U$ 
in that proposition. 
Using lemma \ref{cor:fixedpt} and the map $\Phi$ of lemma \ref{lem:diffeo}, it follows that 
$F^L:=\Phi(I)$ is a submanifold 
 {of $\widetilde{N}$} consisting  of fixed points of $L$. 
\end{proof}

\begin{remark}
{
We have a short exact sequence of cochain complexes
$$\{0\}\to \wedge^{\bullet}\h^{\circ}
\to 
\wedge^{\bullet} \g^* 
\to
\frac{\wedge^{\bullet} \g^*}{\wedge^{\bullet}\h^{\circ}}
\to \{0\}$$ with differentials induced by $d_{\g}$; notice that 
the first complex   agrees with the Chevalley-Eilenberg complex of the quotient Lie algebra $\g/\h$. A piece of the corresponding long exact sequence in cohomology reads $H^2(\g)\to H^2\left(\frac{\wedge^{\bullet} \g^*}{\wedge^{\bullet}\h^{\circ}},\overline{d_{\g}}\right) \to H^3(\g/\h)$. In particular, when the Lie algebra cohomology groups  $H^2(\g)$ and $ H^3(\g/\h)$ vanish, the obstruction in theorem \ref{thm:stableDirac2} also vanishes.
}
\end{remark}

\begin{remark}[Comparison with stability of Lie algebroids]
{Recall that every Dirac structure inherits a Lie algebroid structure $d_A$.
As one may expect, the cochain complex appearing in the obstruction in theorem \ref{thm:stableDirac2} does not only depend on the induced Lie algebroid structure of $A$: Indeed, $\g^*/\h^\circ$ has the same dimension as the leaf of the Courant algebroid through the fixed point $p$, by \eqref{eq:isoleaf}.}
{As any Dirac structure near $A$ induces a Lie algebroid structure on $A$ which is near $d_A$, there is a relation with the stability of a fixed point of the Lie algebroid $A$, as in \cite{CrFe}. This relation is reflected in the cohomological obstructions (see also the text below theorem 2 in the introduction and lemma 1.12 of \cite{CrFe}). Recall that the cohomological obstruction from \cite{CrFe} to the stability of a fixed point {of $A$ as a Lie algebroid} is given by $H^1(\g, T_pM)$. Here the action of $\g$ on $T_pM$ for $x\in \g, v \in T_pM$ is given by
$$
x\cdot v = [\rho_A(x),v].
$$
For $k\geq 1$, the map
\begin{align*}
\wedge^k\g^\ast &\to \wedge^{k-1} \g^\ast \otimes T_pM \\
\alpha_1 \wedge \dots \wedge \alpha_k& \mapsto \sum_{i=1}^k (-1)^{k-i}\alpha_1\wedge \dots \wedge \widehat{\alpha_i} \wedge \dots \wedge \alpha_k \otimes \rho_{A^\ast}(\alpha_i)
\end{align*}
descends to an injective chain map
\begin{equation}\label{eq:injchainmap}
\frac{\wedge^k \g^\ast}{\wedge^k \h^\circ }\to \wedge^{k-1} \g^\ast \otimes T_pM.
\end{equation}
Note that this map takes values in the subspace $\wedge^{k-1} \g^\ast \otimes T_pN$, where $N$ is the leaf of the Courant algebroid through $p$. The collection of these spaces forms a \emph{subcomplex} of $\wedge^{\bullet} \g^\ast \otimes T_pM$, but we will not use this.
For $k = 2$, the induced map \eqref{eq:injchainmap}} in cohomology relates the cohomological obstructions.

{In general, this map is neither injective, nor surjective, so vanishing of either cohomological obstruction does not imply vanishing of the other. This is to be expected, because while Dirac structures near $A$ are contained in the Lie algebroid structures near $d_A$, the equivalences for Dirac structures only allow to move $p$ along the leaf through $p$ of the Courant algebroid $E$.} 
However, if $\rho:E\to TM$ is surjective, then stability of a fixed point in the realm of Lie algebroids does imply stability of the fixed point of the Dirac structure. This is reflected at the level of obstructions: if $\rho$ is surjective, then the map in \eqref{eq:injchainmap} is injective in cohomology, hence the vanishing of $H^1(\g, T_pM)$ implies the vanishing of the obstruction in theorem \ref{thm:stableDirac2}.
\end{remark}

\begin{remark}
It would be interesting to investigate whether the statement of theorem \ref{thm:stableDirac2} remains true removing the split-signature condition.
\end{remark}

\section{Examples}\label{sec:ex}

In this section we present several examples for theorem \ref{thm:stableDirac2}, about the stability of fixed points of Dirac structures.
All our examples are of the kind we describe in this remark.

\begin{remark}\label{rem:BB*}
{Let $B$ be a Lie algebroid over $M$ and a pick a closed $H\in \Gamma(\wedge^3 B^*)$, yielding a Courant algebroid $(B \oplus B^*)_H$ as in example \ref{rem:double}. Let $\pi\in \Gamma(\wedge^2 B)$ such that 
\begin{equation}\label{eq:pitwist}
  [\pi,\pi]_B+2(\wedge^3\pi^{\sharp})(H)
=0.
\end{equation}
 Then $A:=
\text{graph}(\pi)$ is a Dirac structure, see example \ref{ex:DiracB}.
Let   $p\in M$ be  a fixed point of the Dirac structure $A$, i.e.  $A_p\subset \ker(\rho_B)_p\oplus B^*_p$, or equivalently $\rho_B\circ \pi^{\sharp}=0$.
}

{
One can compute the obstruction as in theorem \ref{thm:stableDirac2}, using $\h=\{0\}\oplus \ker(\rho_B)_p^{\circ}\subset A_p$. Often however we prefer to compute the obstruction using the characterization given in proposition \ref{thm:stableDirac}, since it yields the differential directly, without the need to make explicit the Lie algebra structure of $A_p$.  
A lagrangian complement to $A$ is $B\oplus \{0\}$, which by the pairing is identified with $A^*$.
Notice that the differential on $\Gamma(\wedge^{\bullet}A^*)\cong \Gamma(\wedge^{\bullet}B)$ 
is $d_B=[\pi,\,\cdot\,]_B+
 (\wedge^2\pi^{\sharp}\otimes \text{id})(H)(\,\cdot\,)$ by \cite[\S 3]{SW}. Hence the obstruction appearing in the theorem is 
\begin{equation}\label{eq:obsttriang}
H^2\left(\frac{\wedge^{\bullet} B_p}{\wedge^{\bullet}\ker((\rho_{B})_p)},\overline{[\pi,\,\cdot\,]_B+
 (\wedge^2\pi^{\sharp}\otimes \text{id})(H)(\,\cdot\,)}
\right),
\end{equation}
where the differential is computed extending to sections of $\wedge^{\bullet}B$ and then evaluating at $p$.
}
\end{remark}
 
 \begin{cor}\label{HtwistedCA}
{Fix a closed 3-form $H \in \Omega^3(M)$. 
Let $\pi \in \Gamma(\wedge^2 TM)$ be an $H$-twisted Poisson structure.
Let $p\in M$ be a point such that  $\pi_p=0$.}

Recall that  $\g:=T^*_pM$ carries a Lie algebra structure, defined by   $[d_pf,d_pg]=d_p\{f,g\}$ where $f,g$ are functions. If its second Chevalley-Eilenberg cohomology vanishes, i.e. $H^2(\g)=0$, then any $H$-twisted Poisson structure nearby $\pi$  
{vanishes along a submanifold of dimension $\dim(H^1(\g))$.}
\end{cor}

For $H = 0$, this recovers \cite[theorem 1.1]{CrFe} for zero-dimensional leaves, and \cite[theorem 1.2]{dufour2005stability} for first order singularities.
Note that the obstruction does not depend on $H$.

\begin{proof}
{By  example \ref{ex:Dirac} we know that 
$A := \text{graph}(\pi)\subset E_H:=(TM \oplus T^\ast M)_H$ is a Dirac structure. Further, $\pi_p=0$ means that  $p\in M$ is a fixed point of the Dirac structure $A$.}
 
{As lagrangian complement to $A$ we choose $B=TM\cong A^*$.
The complex appearing in \eqref{eq:obsttriang} 
is just $\wedge^{\bullet}T_pM$, since the anchor $\left. \rho\right|_{TM} $ is injective. The differential reads $\overline{[\pi,\,\cdot\,]}$ (notice that the second summand in  \eqref{eq:obsttriang} vanishes, since $\pi_p=0$).}

{This is exactly the complex computing the Chevalley-Eilenberg cohomology of the isotropy Lie algebra of $A$ at the point $p$, and this Lie algebra is the one described in the statement of this corollary.}
\end{proof}
\begin{remark}\label{rem:locallyuntwist}
Since the stability problem is local and $H$ locally admits a primitive 2-form $\omega$ around $p\in M$,  by remark \ref{rem:untwist} the Dirac structure $\gr(\pi)$ is isomorphic to a Dirac structure in the standard Courant algebroid. Because $\pi_p = 0$, this Dirac structure is the graph of a Poisson structure near $\pi$ on a neighborhood of $p$ (see \cite[\S 4]{SW}). Corollary \ref{HtwistedCA} is therefore equivalent to \cite[theorem 1.1]{CrFe} for zero-dimensional leaves and \cite[theorem 1.2]{dufour2005stability} for first order singularities.
\end{remark}

\begin{ep}[Cartan-Dirac structure]\label{CartanDirac}
{Let $G$ be a Lie group with a bi-invariant, possibly indefinite metric  $(\,\cdot\,,\,\cdot\,)$ (for instance, a compact Lie group).
The Cartan-Dirac structure on $G$ was introduced in \cite[example 5.2]{SW}, and
 is a Dirac structure in the twisted Courant algebroid $(TG\oplus T^*G)_{-H}$. Here $H$ is the Cartan 3-form, i.e. the bi-invariant 3-form on $G$  which at the unit element reads $H(u,v,w):=\frac{1}{2}([u,v],w )$ for $u,v,w\in \g=T_eG$. 
Explicitly, it is given by 
$A:=\{(v^L-v^R, \frac{1}{2}(v^L+v^R)^{\flat}):v\in \g\}$, where $v^L$ and $v^R$ denote the left-invariant and right-invariant extension, and $\flat$ denotes contraction with the metric.}

With the induced Lie algebroid structure, $A$ is isomorphic (over $Id_G$) to the transformation Lie algebroid associated to the action of $G$ on itself by conjugation. In particular the leaves of the   Cartan-Dirac structure $A$ are the conjugacy classes of $G$. So the unit $e\in G$ is a fixed point of $A$, and the isotropy Lie algebra of $A$ at $e$ is just $\g$.
By corollary \ref{HtwistedCA} we hence know that if $H^2(\g)=0$ then, {for any neighborhood $U\subset G$ of $e$, there exists a neighborhood of $A$ in the space of Dirac structures consisting of $H$-twisted Poisson structures with a fixed point near $p$.}
\end{ep}

In the preceding corollary and example, the twisting 3-forms $H$ could be chosen such that their cohomology classes are nonzero. However, as explained in remark \ref{rem:locallyuntwist}, the only thing that matters is the cohomology class when restricting to a neighborhood of a point. {When the restricted twisting is exact, locally there exist a Dirac complement, by example \ref{ex:DiracB}.}
Below we give an instance {where the twisting is not even locally exact, and  there is no locally defined Dirac complement. In such a case,   \cite[theorem 5.50]{KarandeepStability} does not apply, and one really needs the more general statement we provided in proposition  \ref{thm:stableDirac}.
}

\begin{ep}
Let $M = \mathbb R^4$ with coordinates $(x_1,x_2,x_3,x_4)$, and let $Z$ be the self-crossing hypersurface given by the equation $x_1x_2x_3 = 0$. Let $B_Z$ be the associated $c$-tangent bundle \cite{ctangent}.
This is the Lie algebroid whose sections consist of vector fields on $\mathbb R^4$ which are tangent to $Z$; 
an adapted frame is provided by  $\{e_1,e_2,e_3,e_4\}$, where $\rho_{B_Z}(e_i) = x_i\partial_{x_i}$ for $i = 1,2,3$ and $\rho_{B_Z}(e_4) = \partial_{x_4}$. Let $\{e^1,e^2,e^3,e^4\}$ denote the dual frame of $B^\ast_Z$. 
Then $\pi \in \Gamma(\wedge^2 B_Z)$ given by 
$$
\pi = x_4 e_1 \wedge e_4
$$
satisfies 
\begin{equation}
\label{eq:Htwistpoiss}
[\pi,\pi]_{B_Z} = 2 \wedge^3 \pi^{\sharp}(H),
\end{equation}
{for any\footnote{For instance
$H = e^1\wedge e^2 \wedge e^3$.} choice of $d_{B_Z}$-closed $H \in \Gamma(\wedge^3 B_Z^\ast)$.} 
Note that the right hand side vanishes, as $\pi^{\sharp}:B_Z^\ast \to B_Z$ has rank at most $2$. 
Equation \eqref{eq:Htwistpoiss} 
implies that $A:=\text{graph}(\pi^{\sharp})$
is a Dirac structure in the Courant algebroid  $(B_Z \oplus B^\ast_Z)_H$, and 
 $p = 0\in \mathbb R^4$ is a fixed point since $\pi$ vanishes there. 

To compute the cohomological obstruction, we denote  $$\mathfrak h :=\ker((\rho_{B_Z})_p) = \text{span}_{\mathbb R}\{e_1(p),e_2(p),e_3(p)\}.$$ Then 
the complex appearing in  \eqref{eq:obsttriang}, in the relevant degrees, 
can be identified with
\[
\begin{tikzcd}
\mathbb{R} e_4(p) \arrow{r}{[\pi,\,\cdot\,]_{B_Z}}& \mathfrak h \wedge\mathbb{R} e_4(p) \arrow{r}{[\pi,\,\cdot\,]_{B_Z}} & \wedge^2\mathfrak h \wedge \mathbb{R} e_4(p),
\end{tikzcd}
\]
{by using the decomposition
$(B_Z)_p=\h\oplus \mathbb{R} e_4(p)$.
}
Here, the differential should be interpreted as extending an element $v \in \wedge^i \mathfrak h \wedge e_4(p)$ to a local section $\widetilde{v} \in \Gamma(\wedge^{i+1} B_Z)$, computing $[\pi,\widetilde{v}]_{B_Z}(p)$ and projecting to the subspace given by $\wedge^{i+1} \mathfrak h\wedge \mathbb R e_4(p)$. {The  cohomology of the above complex at $\mathfrak h\wedge\mathbb{R} e_4(p)$ vanishes, as one sees using the fact that the frame $\{e^1,e^2,e^3,e^4\}$  of $B^\ast_Z$ consists of pairwise commuting sections. Thus 
theorem \ref{thm:stableDirac2} 
implies that any Dirac structure in $(B_Z \oplus B_Z^\ast)_H$ close to $\text{graph}(\pi^{\sharp})$ has a fixed point near $0$.
}
\end{ep}

\begin{remark}\label{rem:stablesection}
In view of remark \ref{rem:stabplus}, notice that $A_p = (B^\ast_Z)_p$. However, the vanishing of $\pi$ as a section of $\Gamma(\wedge^2 B_Z)$ is not stable, as the graph of the $c$-bivector field  $\pi_t = \pi + t e_1 \wedge e_2$ is Dirac, but does not coincide with $B_Z^\ast$ at any point {for $t\neq 0$}.
\end{remark}

{
\begin{remark}[On the induced Poisson bivector field] 
Let $B$ be a Lie algebroid over $M$ with anchor $\rho$, let $h$ be a closed 3-form on $M$, 
and let  $\pi\in \Gamma(\wedge^2 B)$ satisfying \eqref{eq:pitwist} for $H:=\rho^*h$. Then $\pi_M:=(\wedge^2\rho)\pi$ is an $h$-twisted Poisson bivector field on $M$.
 If $p$ is a fixed point of $\gr(\pi)$, then $p$ is a fixed point   of $\pi_M$, since $\pi_M^{\sharp}=\rho\circ\pi^{\sharp}\circ\rho^*$.
Therefore, assuming for simplicity that the anchor $\rho$ is an isomorphism on an open dense set of $M$, theorem \ref{thm:stableDirac2} implies the following:  
if \eqref{eq:obsttriang} vanishes, then any $h$-twisted Poisson bivector field nearby $\pi_M$ which can be lifted to $B$, has a fixed point nearby $p$. For instance, when $B$ is
the $c$-tangent bundle associated to a self-crossing hypersurface $Z$, this is a statement about $h$-twisted Poisson bivector field nearby $\pi_M$ which are tangent to $Z$. See \cite[\S 5.1.5]{KarandeepStability} for an example in the case $h=0$.
\end{remark}
}

\appendix
\section{}\label{appA}
    In this appendix we prove two lemmas needed in the body of the paper.

Assume the setting and notation introduced at the beginning of \S \ref{sec:stabDirac}. In  item ii) there, we stated that  a certain   subspace $W$ is  a closed $L_\infty[1]$-subalgebra of the $L_\infty[1]$-algebra $V$ introduced there. We now prove this fact.

  \begin{lemma}\label{lem:subalg}
    Let $W^i \subset V^i$ be defined by 
    $$
    W^i:= \{\Lambda \in \Gamma(\wedge^{i+2} A^\ast) \mid \Lambda_p \in \wedge^{i+2} \ker((\rho_{A^\ast})_p)\}.
    $$
    Then
    \begin{itemize}
        \item [1)] $W^i\subset V^i$ is a closed subspace for $i = -1,0,1$,
        \item[2)] $W = \bigoplus_{i=-2}^{\infty}W^i$ is a $L_\infty[1]$-subalgebra of $(\Gamma(\wedge^{\bullet+2} A^\ast),\{\mu_k\}_{1\leq k\leq 3})$.
    \end{itemize}
    \end{lemma}
    \begin{proof}
    \begin{itemize}\item[]
    \item[1)] Recall that 
    the evaluation map $\text{ev}_p:\Gamma(\wedge^{i+2} A^\ast) \to \wedge^{i+2} A^\ast_p$ is continuous when the left hand side is equipped with the $C^k$-topology for some $k\geq 0$. As $\wedge^{i+2} \ker((\rho_{A^\ast})_p)\subset \wedge^{i+2} A^\ast_p$ is closed, it follows that $W^i = \text{ev}_p^{-1}(\wedge^{i+2} \ker((\rho_{A^\ast})_p))$ is closed.
    \item[2)]
{Notice that $W$ is invariant under wedge product and the multibrackets $\mu_1, \mu_2,\mu_3$ are graded derivations in each entry \cite[Remark B.2]{defSF}. Because of this
it is sufficient to show that  $W^{-2}$ and $W^{-1}$ -- the degree components that generate $W$ --
are closed under the multibrackets.
Since  $W^{-2}=V^{-2}$ and by degree reasons, 
we are actually reduced to showing that
for $f\in W^{-2}=C^{\infty}(M)$ and
 $X,Y\in W^{-1}= 
 \text{ev}_p^{-1}( \ker((\rho_{A^\ast})_p))$:
$$
    \mu_1(f) \in W^{-1},\;\;\;\; \mu_1(X) \in W^0,\;\;\;\; \mu_2(X,Y) \in W^{-1}.
    $$}

 {   
 We already showed the first two statements in the proof of 
lemma \ref{lem:hypo} ii).}

    We   show that $\mu_2(X,Y) \in W^{-1}$. As 
    $$
    \mu_2(X,Y) = \oo X,Y \cc - \Psi(X,Y,\,\cdot\,),
    $$
    we note that
    \begin{align*}
    \rho_{A^\ast}(\mu_2(X,Y)) &= \rho(\oo X,Y\cc) - \rho_A(\Psi(X,Y,\,\cdot\,))\\
    &= [\rho(X),\rho(Y)] - \rho_A(\Psi(X,Y,\,\cdot\,)).
    \end{align*}
    Evaluating the right hand side in $p$, the first term vanishes because it is the Lie bracket of vector fields $\rho(X)$ and $\rho(Y)$ which vanish in $p$, while the second term vanishes because $\rho_A$ vanishes at $p$. This shows that $\mu_2(X,Y) \in W^{-1}$. 
    \end{itemize}
    \end{proof}
 
The following statement is needed in the proof of lemma \ref{lem:diffeo}. We include a proof for completeness.

\begin{lemma}\label{lem:spray}
{Let $N$ be a manifold, $p$ a point, and consider a linear map $X\colon T_pN\to \mathfrak{X}_c(N)$ to the compactly supported vector fields, mapping each vector $v\in T_pN$ to a vector field $X^v$ extending it (i.e. $X^v(p)=v$). Denote by $\phi^1_{X^v}$ the time-$1$ flow of $X^v$. Then the map
$$\Psi\colon T_pN\to N,\; v\to \phi^1_{X^v}(p),$$
when restricted to a suitable neighborhood of the origin, is a diffeomorphism onto its image.}
\end{lemma}

\begin{proof}
We can express $\Psi$ in terms of the vector field $Y$ on $N\times T_pN$ defined by $Y(q,v):=X^v(q)$, as follows:  $\Psi(v)=pr_N (\phi^1_Y(p,v))$. This description implies that
the map $\Psi$ is smooth.
The derivative of $\Psi$ at the origin is $Id_{T_pM}$, as one computes
$$(d_0\Psi)(v)=\left.\frac{d}{dt}\right|_{t=0}\Psi(tv)=X^v(p)=v$$
using $\phi^1_{X^{tv}}=\phi^1_{tX^v}=\phi^t_{X^v}$.
Hence the statement follows from the inverse function theorem.
\end{proof}

\bibliographystyle{habbrv} 

\end{document}